\documentclass[11pt]{article}
\usepackage[utf8]{inputenc}
\usepackage[whole]{bxcjkjatype}
\usepackage{amsmath,amssymb,bbm}
\usepackage{graphicx}
\usepackage{mathtools,stmaryrd,newunicodechar}
\usepackage{algorithm,algorithmicx,algpseudocode}
\usepackage{natbib,hyperref}
\usepackage{geometry}
\usepackage{amsthm}
\usepackage{authblk}
\usepackage{fourier}
\usepackage{subcaption}
\usepackage{bbm}
\usepackage{bm}
\usepackage{enumerate}
\usepackage{autobreak}
\usepackage{mdframed}
\usepackage{pdfpages}

\hypersetup{
    colorlinks = true,
    citecolor=blue,
    linkcolor=blue
}

\newcommand{\critset}{\mathcal{C}^{\star}}

\newcommand{\gm}{\overline{\varphi}}

\newcommand{\setH}{\mathsf{H}}
\newcommand{\simplex}{\mathcal{S}^{\circ}}

\newcommand{\curves}{\mathsf{C}}

\newcommand{\zero}{\mathsf{Z}}

\newcommand{\ku}{\kappa\!_u}
\newcommand{\kv}{\kappa\!_v}
\newcommand{\gu}{\gamma\!_u}
\newcommand{\gv}{\gamma\!_v}

\newtheorem{theorem}{Theorem}

\newtheorem{lemma}{Lemma}

\title{On Mixtures of Three Homoscedastic Gaussian Densities: \\ 
An Unconditional Sharper Bound on the Number of Modes}

\author[1,2,3]{Akifumi Okuno\thanks{\url{okuno@ism.ac.jp} (corresponding author)}}
\author[4]{Yutaro Kabata}
\affil[1]{Institute of Statistical Mathematics}
\affil[2]{The Graduate University for Advanced Studies, SOKENDAI}
\affil[3]{RIKEN}
\affil[4]{Kagoshima University}

\begin{document}


\maketitle

\begin{abstract}
It is known that a mixture of three homoscedastic multivariate Gaussian densities whose centers form an equilateral triangle can have four modes, owing to the emergence of a ``ghost'' mode at the center. Nevertheless, obtaining a sharp upper bound on the number of modes remains open even in this seemingly simple setting. 
The best previously available upper bounds applicable to the homoscedastic three-component setting were $42592$ and, more recently, $72$. These bounds were derived for substantially more general classes of Gaussian mixtures and are therefore not optimized for the present setting. Moreover, they are conditional on the assumption that the modal set is finite. In this paper, as a sharper bound, we prove that every mixture of three homoscedastic Gaussian densities has at most $8$ modes, without imposing any finiteness or non-degeneracy assumption a priori. To the best of our knowledge, this is the first unconditional finiteness result for the modal set that holds uniformly over the full class of multivariate three-component homoscedastic Gaussian mixtures. In particular, it covers genuinely multivariate asymmetric configurations with arbitrary non-collinear centers and arbitrary positive mixture weights.
\end{abstract}

\section{Introduction}
Gaussian mixture models are among the most widely used tools in modern statistics, supporting a broad range of applications, including clustering and density estimation~\citep{mclachlan1988mixture,mclachlan2000finite}. These applications often rely on the modal structure of a Gaussian mixture density, where modes are understood as local maxima of the density. In clustering, in particular, each mode is commonly interpreted as representing a distinct cluster or underlying group. Consequently, when the number of mixture components is specified as $k$, it is natural to expect that the number of modes should not exceed $k$.

In the univariate case, the number of modes of a Gaussian mixture with $k$ components is at most $k$, even when the component variances are different~\citep[][Theorem~2]{perpinan2003number}. The proof relies on scale-space theory, namely on the behavior of modes under Gaussian blurring and deblurring. Univariate Gaussian blurring does not increase the number of modes. This monotonicity is special to the univariate setting; in the multivariate case, Gaussian blurring may create new modes. A two-dimensional counterexample was given by \citet{Lifshitz1990multiresolution}, and a local theoretical analysis was later developed by \citet{damon1995local}. As discussed in Section~3.2 of \citet{perpinan2003number}, this failure of scale-space monotonicity is one reason why multivariate mode counting is substantially more difficult.

Indeed, a multivariate Gaussian mixture with $k$ components can exhibit more than $k$ modes. When the component covariance matrices differ, such examples are relatively easy to construct. For instance, if two elongated bivariate Gaussian components are arranged to intersect at a suitable angle, an additional mode can emerge in the region where their tails overlap, yielding three modes from only $k=2$ components; see, for example, Figure~1 of \citet{kabata2025singularities} or \citet{amendola2019maximum}. More generally, \citet{ray2012upper} proved that the maximum number of modes of a two-component Gaussian mixture in $\mathbb{R}^d$ is exactly $d+1$. By contrast, for mixtures with three or more components, several explicit constructions provide lower bounds~\citep{ray2005topography,perpinan2003isotropic,perpinan2003number,edelsbrunner2013isotropic}, but the sharp maximum remains unknown.

General upper bounds were established by \citet{amendola2019maximum}, and were recently improved by \citet{nguyen2026bounds}. These works first bound nondegenerate critical points using Pfaffian or fewnomial techniques~\citep{khovanskii1980class,khovanskii1991fewnomials}, and then obtain mode-count bounds under a finite-modal-set assumption via Morse-theoretic or tilting arguments. Thus, for the purpose of obtaining an unconditional mode bound, the remaining difficulty is not only to sharpen the numerical estimate but also to rule out degenerate or non-isolated critical/modal sets. Moreover, the resulting general-purpose bounds remain rather loose in the special homoscedastic three-component setting.


Because determining the maximum number of modes in full generality appears prohibitively difficult, we focus on the homoscedastic setting~\citep{perpinan2003isotropic,edelsbrunner2013isotropic}, where all components share a common covariance matrix. Besides being considerably more tractable mathematically, homoscedastic Gaussian mixtures form an important parsimonious model in statistical practice, requiring substantially fewer parameters than general heteroscedastic mixtures.
More specifically, we study three-component mixtures. 
In this setting, the general upper bounds of \citet{amendola2019maximum} and \citet{nguyen2026bounds} give the finite-modal-set bounds $42592$ and $72$, respectively. These results do not by themselves establish finiteness of the modal set for the mixture under consideration. We prove this finiteness directly and derive a substantially sharper unconditional upper bound of $8$ modes (Theorem~\ref{theo:bound_modes}).

The remainder of this paper is organized as follows. Section~\ref{sec:problem} presents the formal problem setting and reviews existing bounds on the number of modes. Section~\ref{sec:number_of_modes} derives a sharper upper bound under the assumption-free setting. 
Section~\ref{sec:conclusion} concludes this paper.

\section{The Homoscedastic Gaussian Mixture Mode Problem}
\label{sec:problem}

Let $d \in \mathbb{N}$ denote the dimension, and let $k \in \mathbb{N}$ denote the number of Gaussian components. In this section, we allow general $k$ in order to review the historical background, although the main results of this paper concern the case $k=3$. 
Let $\bm{\mu}_1,\ldots,\bm{\mu}_k \in \mathbb{R}^d$ be the component centers. Let $w_1,\ldots,w_k$ be arbitrary non-negative weights satisfying $\sum_j w_j=1$. 
For a homoscedastic Gaussian mixture with common covariance matrix $\bm{\Sigma}$, the change of variables $\tilde{\bm{x}}=\bm{\Sigma}^{-1/2}\bm{x}$ and $\tilde{\bm{\mu}}_j=\bm{\Sigma}^{-1/2}\bm{\mu}_j$ gives a one-to-one correspondence between the modes of the original mixture and those of the standardized mixture. Hence, without loss of generality, we take $\bm{\Sigma}=\bm{I}$. With the standard $d$-variate Gaussian density $\varphi$, this study focuses on the maximum number of modes in its mixture
\[
\gm(\bm{x}) = \sum_j w_j \varphi_j(\bm{x}),
\quad
\varphi_j(\bm{x}) = \varphi(\bm{x} - \bm{\mu}_j).
\]

Throughout this paper, the term \emph{mode} refers to a local maximizer of $\gm$. If the set of modes is not finite (for example, if $\gm$ possesses a non-isolated mode), we define the number of modes to be infinite. 
Existing results on lower and upper bounds for the maximum number of modes relevant to this problem are reviewed in Sections~\ref{subsec:lower} and~\ref{subsec:upper}, respectively.

\subsection{Existing lower bounds for homoscedastic mixtures} 
\label{subsec:lower}
\citet{perpinan2003isotropic} describe a private communication in which J.~J.~Duistermaat provided an example showing that $\gm(\bm{x})$ with $k=3$ Gaussian components, arranged at the vertices of an equilateral triangle, can produce $4$ modes. Also see Example 2.1.4 of \citet{amendola2017algebraic} or Figure~3 of \citet{amendola2019maximum}. 

Along the same lines, the equilateral triangle~(i.e., $2$-simplex) can be generalized to $(k-1)$-simplex. \citet{edelsbrunner2013isotropic} studied mixtures of homoscedastic and isotropic $k$ Gaussian components whose centers form a $(k-1)$-simplex embedded in an affine subspace of dimension $k-1$ with equal weights $w_1=w_2=\cdots=w_k=1/k$, and they showed that such highly symmetric configurations can have $k+1$ modes. Using this construction, \citet{edelsbrunner2013isotropic} further proved that $k=3^m$ ($m \in \mathbb{N}$) components mixture in high dimension can contain approximately $k^{1.262}$ modes. 
Although \citet{amendola2019maximum} and \citet{nguyen2026bounds} also discuss lower bounds in more general settings, we do not pursue them further here, as lower-bound constructions are not the main focus of this paper.

\subsection{Existing upper bounds for homoscedastic mixtures} 
\label{subsec:upper}

\citet{ray2012upper} proves that $2$-component Gaussian mixture in $\mathbb{R}^d$ has at most $d+1$ modes. For more than $2$ components, Section~3.2 and Corollary~5.3 of \citet{wallace2013critical} proves that the modal set is finite under strong symmetry assumptions; the mixture there is restricted to equally weighted and isotropic, and its component centers $\bm{\mu}_1,\bm{\mu}_2,\ldots,\bm{\mu}_k$ are located at the vertices of a regular $(k-1)$-simplex, as in \citet{edelsbrunner2013isotropic}. Moreover, Corollary~5.7 of \citet{wallace2013critical} establishes finiteness for proportional-covariance mixtures whose component centers are collinear. This case is essentially univariate and will be treated as a degenerate case in Section~\ref{subsec:degeneracy} in our setting.

Turning from qualitative finiteness to quantitative upper bounds, \citet{amendola2019maximum} Theorem~10 proves that the maximum number of modes in the $d$-variate Gaussian mixture of $k$ components is upper bounded by $2^{d+\binom{k}{2}}(5+3d)^k$, by \emph{assuming} that every mixture of $k$ Gaussians in $\mathbb{R}^d$ has finitely many modes. For a mixture of three Gaussian densities, the upper bound in our setting is evaluated as
$2^{d+3}(5+3d)^{3}$, which is no less than $2^{2+3}(5+3 \cdot 2)^{3}=42592$. 
\citet{nguyen2026bounds} refined the bounds of \citet{amendola2019maximum} for general settings. In particular, \citet{nguyen2026bounds} derived sharper upper bounds for homoscedastic Gaussian mixtures. Substituting our setting, corresponding to $r=2$ and $k=3$, into Corollary~3.20 yields the upper bound $72$, under a priori finite-modal-set assumption. 

Although no formal proof is known, B.~Sturmfels conjectured that the maximum number of modes of a general Gaussian mixture is $\binom{d+k-1}{d}$, which equals $6$ when $d=2$ and $k=3$; see \citet{amendola2019maximum} for historical context. This conjecture concerns the fully general setting, including heteroscedastic mixtures. By contrast, in the homoscedastic setting considered here, the largest number of modes currently known for three-component mixtures is $4$, attained by the classical equilateral-triangle construction of \citet{perpinan2003isotropic}. To the best of our knowledge, no construction with $5$ or more modes has been found in the homoscedastic three-component setting.

\section{Bounding the Number of Modes for Three-Component Mixture}
\label{sec:number_of_modes}

The mode-count bounds discussed above are conditional on the finiteness of the modal set, while the underlying critical-point bounds count non-degenerate critical points. In this section, we prove a sharper upper bound without imposing either a finiteness assumption or a non-degeneracy assumption.

We first state our theorem regarding the number of critical points.

\begin{theorem}
\label{theo:bound_critical_points}
Let $d\ge 1$ and $k=3$. Let $\bm{\mu}_1,\bm{\mu}_2,\bm{\mu}_3\in\mathbb{R}^d$ be arbitrary centers, and let $w_1,w_2,w_3$ be arbitrary non-negative weights satisfying $\sum_{j} w_j=1$. Then the Gaussian mixture density $\gm$ has at most $15$ critical points. 
In particular, every critical point is isolated. 
Non-degeneracy assumption on the critical points is not imposed a priori.
\end{theorem}

If we additionally assume that $\gm$ is a Morse function~\citep{milnor1963morse}, i.e., Hessians of all the critical points are non-degenerate, then Proposition~3.4 of \citet{nguyen2026bounds} gives that the number of modes for the Morse Gaussian mixture is no more than $(15+1)/2=8$. The Morse assumption can be removed using the finite-mode transfer result of Proposition~3.7 therein. Let $\mathcal{G}$ denote the class of homoscedastic Gaussian mixture densities with at most three effective components. This class is closed under normalized exponential tilting, since $\exp(\bm{c}^{\top}\bm{x})\varphi(\bm{x}-\bm{\mu}_j) =\exp(\bm{c}^{\top}\bm{\mu}_j+\|\bm{c}\|^2/2) \varphi(\bm{x}-(\bm{\mu}_j+\bm{c}))$; thus, normalization shifts all centers by $\bm{c}$ and changes the weights proportionally to $w_j\exp(\bm{c}^{\top}\bm{\mu}_j)$.

By Theorem~\ref{theo:bound_critical_points}, every density in $\mathcal{G}$ has at most $15$ non-degenerate critical points, and the modal set of $\gm$ is finite. Proposition~3.7 of \citet{nguyen2026bounds} therefore applies with $U=15$ and yields at most $(15+1)/2=8$ modes, proving Theorem~\ref{theo:bound_modes}.

\begin{theorem}
\label{theo:bound_modes}
Let the notation and assumptions be as in
Theorem~\ref{theo:bound_critical_points}. Then $\gm$ has at most $8$
modes.
\end{theorem}

To the best of our knowledge, this is the first unconditional finiteness result for the modal set that holds uniformly over the full class of multivariate three-component homoscedastic Gaussian mixtures. It covers genuinely multivariate asymmetric configurations with arbitrary non-collinear centers and arbitrary positive mixture weights.

The remainder of this section is devoted to the proof of Theorem~\ref{theo:bound_critical_points}. Section~\ref{sec:number_of_zeros_in_exponential_polynomials} provides an analytic tool for counting zeros of exponential polynomials. Using this tool, we first handle the degenerate cases in Section~\ref{subsec:degeneracy}. Section~\ref{subsec:estimating_equation} then derives the equations used to identify critical points, and Section~\ref{subsec:recasting_problems} reformulates the problem as one of counting intersections of two curves. Finally, Section~\ref{subsec:counting_the_number_of_intersections} bounds the number of such intersections, thereby proving Theorem~\ref{theo:bound_critical_points}.

\subsection{Zeros of exponential polynomials}
\label{sec:number_of_zeros_in_exponential_polynomials}

In this section, we present an analytic tool that is essential for our proof. The following lemma gives an upper bound on the number of zeros of exponential polynomials.

\begin{lemma}
\label{lem:exp_zeros}
Let $\lambda_0<\cdots<\lambda_n$ be distinct real numbers, and let $p_0,\ldots,p_n$ be real polynomials, not all identically zero. 
$\deg p_j$ denotes the degree of the polynomial $p_j$, and we use the convention \(\deg 0=-1\). 
Then the number of real zeros, counted with multiplicity, of $f(s)=\sum_{j=0}^n p_j(s) \exp(\lambda_j s)$ is at most $\sum_{j=0}^n(\deg p_j+1)-1$. Moreover, every real zero of $f$ is isolated.
\end{lemma}

The zero-counting result is covered by the general theory of extended complete Tchebycheff systems~\citep[ECT-systems;][]{karlin1966tchebycheff} and by disconjugacy theory for linear differential equations~\citep{coppel1971disconjugacy}. Since our purpose is limited to counting real zeros of this particular class of exponential polynomials, we instead give a short elementary proof that makes the precise degree-dependent bound explicit.

\begin{proof}[Proof of Lemma~\ref{lem:exp_zeros}]
Since $f$ is real analytic and not identically zero, every real zero of $f$ is isolated by the identity theorem. 
By the convention \(\deg 0=-1\), an identically zero function contributes nothing to the right-hand side. Hence we may delete all identically zero terms and relabel the remaining terms. Thus, in the proof below, we may assume without loss of generality that all \(p_j\) are nonzero.

We induct on the number $n+1$ of exponential terms. 
\begin{itemize}
\item When $n=0$, the
zeros of $p_0(s)e^{\lambda_0s}$ are precisely the zeros of $p_0$, so
the assertion is immediate.
\item Suppose that $n\ge1$. Multiplying by $e^{-\lambda_0s}$ does not change
the zeros or their multiplicities, so set
$h(s)=e^{-\lambda_0s}f(s)=p_0(s)+\sum_{j=1}^n p_j(s)e^{\delta_j s}$,
where $\delta_j=\lambda_j-\lambda_0>0$. Let
$d_0=\deg p_0+1$. Differentiating $h$ exactly $d_0$ times annihilates
$p_0$ and gives
$h^{(d_0)}(s)=\sum_{j=1}^n q_j(s)e^{\delta_j s}$, where
$q_j=(d/ds+\delta_j)^{d_0}p_j$. Since $\delta_j\ne0$, the leading
coefficient of $q_j$ is that of $p_j$ multiplied by
$\delta_j^{d_0}$, and hence $\deg q_j=\deg p_j$. 
Write $\zero(u)$ for the number of real zeros of $u$, counted
with multiplicity. Repeated application of Rolle's theorem gives
$\zero(h)\le \zero(h^{(d_0)})+d_0$; indeed, a zero of
multiplicity $m$ produces a zero of multiplicity at least $m-1$ in the
derivative, and there is an additional zero of the derivative between
each pair of consecutive distinct zeros. The induction hypothesis
therefore yields
\[
\zero(f)
\, = \,
\zero(h)
\, \le \,
\zero(h^{(d_0)})+d_0
\, \le \,
\left(\sum_{j=1}^n(\deg q_j+1)-1\right)+d_0
\, = \,
\sum_{j=0}^n(\deg p_j+1)-1.
\]
\end{itemize}
More precisely, the same inequality applies to any finite collection of zeros of $h$; hence, if $h$ had more than $\zero(h^{(d_0)})+d_0$ zeros, one could select a finite subcollection contradicting Rolle's theorem. 
Therefore, the assertion is proved.
\end{proof}

We first illustrate the usefulness of the analytic tool. Lemma~\ref{lem:exp_zeros} immediately implies the following result.

\begin{lemma}
\label{lem:critical_univariate}
Every critical point of a univariate homoscedastic Gaussian mixture with three components is isolated. Moreover, such a mixture has at most $5$ critical points.
\end{lemma}

\begin{proof}
Since $\gm'(x)=-\sum_j w_j (x-\mu_j)\varphi(x-\mu_j)$, the critical points of $\gm$ are precisely the real zeros of $f(s)=-\exp(-s^2/2)\sum_j w_j(s-\mu_j)\exp(-\mu_j^2/2)\exp(\mu_j s)$ after removing a constant factor. Since $\exp(-s^2/2) \ne 0$ for all $s$, 
the function $f$ is an exponential polynomial with at most three distinct exponential terms $p_j(s)=w_j\exp(-\mu_j^2/2)(s-\mu_j)$ ($j=1,2,3$); each polynomial is of degree at most one. Hence Lemma~\ref{lem:exp_zeros} gives
$\#\{x \in \mathbb{R} \mid \gm'(x)=0\} = \#\{s \in \mathbb{R} \mid f(s)=0\} \le 3(1+1)-1=5$.
The assertion follows.
\end{proof}

\subsection{Excluding degenerate cases}
\label{subsec:degeneracy}

We now return to the proof of Theorem~\ref{theo:bound_critical_points}. 

If \(d=1\), the assertion follows directly from Lemma~\ref{lem:critical_univariate}. Hence assume \(d\ge2\). Remove zero-weight components and merge components with coincident centers. If only one positive effective component remains, the density is a single Gaussian and has exactly one critical point. Otherwise, if the positive effective centers are collinear, then after a translation and an orthogonal change of coordinates they can be written as $\bm{\mu}_j=(m_j,0,\ldots,0)$. 
Writing $\bm{x}=(y,\bm{z})$, we have
\[
\gm(y,\bm{z})=\varphi^{(d-1)}(\bm{z})h(y),
\quad
h(y)=\sum_j \widetilde w_j\varphi^{(1)}(y-m_j).
\]
$\varphi^{(d-1)},\varphi^{(1)}$ denote the $(d-1)$-variate and univariate standard Gaussian densities, respectively. Since 
\[
\frac{\partial \gm(y,\bm{z})}{\partial x}
= 
\left(\varphi^{(d-1)}(\bm{z})h'(y) \, , \, -\bm{z}\varphi^{(d-1)}(\bm{z})h(y) \right),
\]
and since $h(y)>0$ and $\varphi^{(d-1)}(\bm{z})>0$, the critical points of $\gm$ are exactly the points $(y,\bm{0})$ with $h'(y)=0$. Therefore this case is reduced to a univariate homoscedastic Gaussian mixture
with at most three components, and Lemma~\ref{lem:critical_univariate} gives
at most $5$ critical points.

Therefore, it remains to consider the genuinely three-component and genuinely multivariate case. Throughout the proof below, we assume $d\ge 2$, 
\begin{enumerate}[{\quad (C1)}]
\item the three centers $\bm{\mu}_1,\bm{\mu}_2,\bm{\mu}_3$ are distinct and not collinear, and
\item the mixture weights $w_1,w_2,w_3$ are strictly positive (i.e., nonzero).
\end{enumerate}

\subsection{Critical-point equations}

\label{subsec:estimating_equation}

To identify the critical points in the Gaussian mixture $\gm(\bm{x})$, we first derive the critical-point equations that they must satisfy. 

Since the modes of a Gaussian mixture depend only on the relative positions of the component centers, we may translate the coordinate system without loss of generality. Thus, we consider the case
\[
\bm{\mu}_1 = \bm{0},
\quad
\bm{\mu}_2 = \bm{u},
\quad
\bm{\mu}_3 = \bm{v},
\]
where Assumption~(C1) implies that $\bm{u},\bm{v} \in \mathbb{R}^d$ are nonzero and linearly independent. We define
\[
\gu = \|\bm{u}\|_2^2/2,
\quad
\gv = \|\bm{v}\|_2^2/2,
\quad
\psi = \langle \bm{u},\bm{v} \rangle.
\]
Moreover, we may assume without loss of generality that
\begin{align}
\psi > 0.
\label{eq:positivity_of_psi}
\end{align}
This condition means that the angle between $\bm{u}$ and $\bm{v}$ is acute. It entails no loss of generality because, for any three non-collinear centers, one can choose one of them as the reference point so that the two vectors from that point to the other two centers form an acute angle.

Also, searching for critical points of \( \gm(\bm{x}) \) over the entire space \( \mathbb{R}^d \) is inefficient. Since $\partial \gm(\bm{x})/\partial \bm{x}
=
-\sum_j w_j(\bm{x}-\bm{\mu}_j)
\varphi(\bm{x}-\bm{\mu}_j)$, every critical point $\bm{x}_{\star}$ satisfies
\begin{align}
\bm{x}_{\star}
=
\sum_j \rho_j(\bm{x}_{\star})\bm{\mu}_j,
\quad
\rho_j(\bm{x}_{\star})
=
\frac{
w_j \, \varphi(\bm{x}_{\star}-\bm{\mu}_j)
}{
\sum_{k} w_{k} \, \varphi(\bm{x}_{\star}-\bm{\mu}_{k})
}.
\label{eq:critical_point_barycentric}
\end{align}
Under Assumption~(C2), \(w_j>0\) for every \(j\), and the Gaussian
density is strictly positive. Hence $\rho_j(\bm{x}_{\star})>0$ for every $j$ and $\sum_j \rho_j(\bm{x}_{\star})=1$. Therefore, to enumerate all the critical points, it suffices to restrict the search to the relative interior of the convex hull of the set $\{\bm{\mu}_1,\bm{\mu}_2,\bm{\mu}_3\}=\{\bm{0},\bm{u},\bm{v}\}$, as stated in Lemma~\ref{lem:modes_in_span}.

\begin{lemma}
\label{lem:modes_in_span}
Every critical point of $\gm(\bm{x})$ lies within the set 
$\setH := \{\alpha \bm{u} + \beta \bm{v} \mid (\alpha, \beta) \in \simplex_2\}$, where $\simplex_2:=\{(\alpha,\beta) \mid \alpha>0,\beta>0, \alpha+\beta<1\}$ is an open $2$-simplex.
\end{lemma}

Since $\bm{u}$ and $\bm{v}$ are linearly independent, $\setH$ is a subset of the $2$-dimensional plane embedded in $\mathbb{R}^d$. 
It is therefore sufficient to analyze essentially a \( 2 \)-dimensional Gaussian mixture in place of the original \( d \)-dimensional one. 
Consequently, all the critical points can be expressed as
\begin{align}
\bm{x} = \alpha \bm{u} + \beta \bm{v} \in \setH,
\label{eq:x_linear}
\end{align}
where $(\alpha,\beta) \in \simplex_2$ are parameters. 
By virtue of the linear independence of $\bm{u}$ and $\bm{v}$, the expression \eqref{eq:x_linear} is unique: $\alpha \bm{u} + \beta \bm{v}=\alpha' \bm{u} + \beta' \bm{v}$ if and only if $(\alpha,\beta)=(\alpha',\beta')$.

The expression \eqref{eq:x_linear} corresponds to the ridgeline representation in Corollary~1 of \citet{ray2005topography}, which in turn cites the original work of \citet{perpinan2003number}. While their formulation includes boundary points of the simplex to cover degenerate cases such as vanishing mixture weights, these essentially univariate cases have already been treated separately in Section~\ref{subsec:degeneracy}. We therefore restrict attention to the relative interior of the simplex.

Under this parameterization, the criticality equation \( \partial \gm(\bm{x})/\partial \bm{x} = \bm{0} \) reduces to a system of simultaneous equations in \((\alpha, \beta)\), as presented in Lemma~\ref{lem:critical}.

\begin{lemma}
\label{lem:critical}
The point \( \bm{x} = \alpha \bm{u} + \beta \bm{v} \in \setH \) is a critical point of \( \gm(\bm{x}) \) if and only if \( (\alpha, \beta) \in \simplex_2 \) satisfies the following system of simultaneous equations:
\begin{align}
\dfrac{\alpha}{1-\alpha-\beta}
=
\frac{w_2}{w_1}
\exp\left(L_u(\alpha,\beta)\right),
\quad
\dfrac{\beta}{1-\alpha-\beta}
=
\frac{w_3}{w_1}
\exp\left( L_v(\alpha,\beta)\right),
\label{eq:simultaneous_simpler}
\end{align}
where the exponent functions $L_u,L_v$ are defined as follows:
\begin{align}
L_u(\alpha,\beta)
=
2\alpha \gu + \beta \psi - \gu, 
\quad
L_v(\alpha,\beta)
=
\alpha \psi + 2\beta \gv - \gv.
\label{eq:def_Lu_Lv}
\end{align}
\end{lemma}

\begin{proof}[Proof of Lemma~\ref{lem:critical}] 
Let $\bm{x} = \alpha \bm{u}+\beta \bm{v}$. 
Since every critical point satisfies the fixed-point equation~\eqref{eq:critical_point_barycentric}, 
the parametrization $\bm{\mu}_1=\bm{0}$, $\bm{\mu}_2=\bm{u}$, and $\bm{\mu}_3=\bm{v}$ gives
\[
    \bm{x}= \rho_2(\bm{x}) \bm{u}+\rho_3(\bm{x}) \bm{v}. 
\]
Using the linear independence of $\bm{u}$ and $\bm{v}$, we obtain $\rho_2(\bm{x})=\alpha$, $\rho_3(\bm{x})=\beta$, and $\rho_1(\bm{x})=1-\alpha-\beta$. Thus, the criticality condition is equivalent to
\begin{align}
\frac{\alpha}{1-\alpha-\beta}
=
\frac{\rho_2(\bm{x})}{\rho_1(\bm{x})}
=
\frac{w_2}{w_1}
\frac{\varphi(\bm{x}-\bm{u})}{\varphi(\bm{x})},
\qquad
\frac{\beta}{1-\alpha-\beta}
=
\frac{\rho_3(\bm{x})}{\rho_1(\bm{x})}
=
\frac{w_3}{w_1}
\frac{\varphi(\bm{x}-\bm{v})}{\varphi(\bm{x})}.
\label{eq:critical_point_weights_ratio}
\end{align}
Since $\varphi(\bm{x}) \propto \exp(-\|\bm{x}\|_2^2/2)$, we have 
\[
\frac{\varphi(\bm{x}-\bm{u})}{\varphi(\bm{x})}
=
\exp\left(
\frac{\|\bm{x}\|_2^2-\|\bm{x}-\bm{u}\|_2^2}{2}
\right)
=
\exp\left(2\alpha\gu+\beta\psi-\gu\right)
=
\exp(L_u(\alpha,\beta)),
\]
and similarly, $\varphi(\bm{x}-\bm{v})/\varphi(\bm{x})=\exp\left(\alpha\psi+2\beta\gv-\gv\right)=\exp(L_v(\alpha,\beta))$. Substituting these expressions into \eqref{eq:critical_point_weights_ratio} proves the assertion. 
\end{proof}

\subsection{Recasting the critical points as intersections of two curves}
\label{subsec:recasting_problems}

To solve the above simultaneous equations \eqref{eq:simultaneous_simpler}, we employ a bijective map 

\begin{align}
    \simplex_2 \ni (\alpha,\beta) 
    \,
    \mapsto
    \,
    (s,t) 
    =
    \left( 
    \log \frac{\alpha}{1-\alpha-\beta}, \,
    \log \frac{\beta}{1-\alpha-\beta}
    \right) 
    \in 
    \mathbb{R}^2,
\label{eq:ab_to_st}
\end{align}
which also has an inverse map
\begin{align}
\mathbb{R}^2 \ni (s,t) 
\,
\mapsto 
\,
(\alpha,\beta)
=
\left(
\frac{e^s}{1+e^s+e^t}, \, 
\frac{e^t}{1+e^s+e^t}
\right)
\in \simplex_2. 
\label{eq:st_to_ab}
\end{align}
Using the above reparameterization and taking the logarithm, the equations \eqref{eq:simultaneous_simpler}, together with \eqref{eq:def_Lu_Lv}, then can be equivalently written as
\begin{align}
P(s,t) := a_0(s) + a_1(s) e^{s} + a_2(s) e^{t} = 0,
\quad 
Q(s,t) := b_0(t) + b_1(t) e^{t} + b_2(t) e^{s} = 0,
\label{eq:reduced_equations}
\end{align}
where 
$\kappa_u = \log(w_1/w_2), \kappa_v=\log(w_1/w_3)$ are log-weight ratios, $a_0(s) = s+\gu+\ku$, 
$a_1(s) = s-\gu+\ku$, 
$a_2(s) = s+\gu+\ku-\psi$ are affine in $s$, and 
$b_0(t) = t+\gv+\kv$, 
$b_1(t) = t-\gv+\kv$, 
$b_2(t) = t+\gv+\kv-\psi$ are affine in $t$. 
See Appendix~\ref{app:derivation} for the derivation in detail. 

We denote by 
\begin{align}
\critset 
:= 
\left\{ (s,t) \in \mathbb{R}^2 \mid P(s,t) = Q(s,t)=0 \right\}
\label{eq:critset}
\end{align}
the set of all critical points expressed in the $(s,t)$-coordinates. 
Because the 
reparameterization $(\alpha,\beta)\mapsto(s,t)$ is a bijection between $\simplex_2$ 
and $\mathbb{R}^2$, and $\bm{x}=\alpha \bm{u}+\beta \bm{v}$ is in turn a bijection onto $\setH$, 
counting the critical points (including modes) of $\gm$ is equivalent to analyzing $\critset$. Geometrically, 
$\critset$ is the intersection of the two planar curves $\curves_P:=\{(s,t) \in \mathbb{R}^2 \mid P(s,t)=0\}$ and $\curves_Q:=\{(s,t) \in \mathbb{R}^2 \mid Q(s,t)=0\}$. Namely, 
\[
    \critset = \curves_P \cap \curves_Q.
\]
The curves $\curves_P, \curves_Q$ and the set of critical points $\critset$ for several settings are illustrated in Figure~\ref{fig:curves}.

\begin{figure}[!ht]
\centering 
\begin{minipage}{0.3\textwidth}
\centering 
\includegraphics[width=0.8\textwidth]{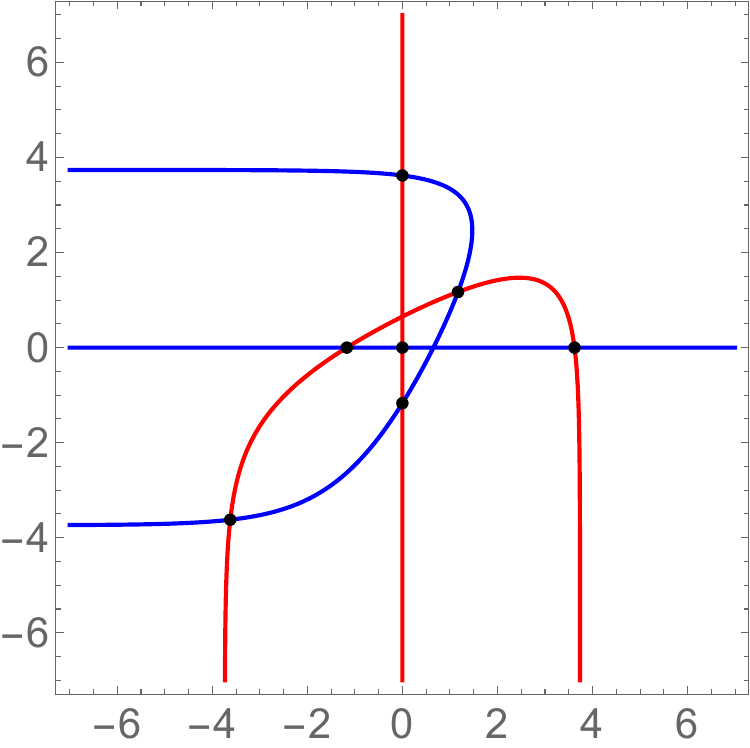}
\subcaption{$\gu=\gv=\psi=3.92$}
\label{fig:curves1}
\end{minipage}
\begin{minipage}{0.3\textwidth}
\centering 
\includegraphics[width=0.8\textwidth]{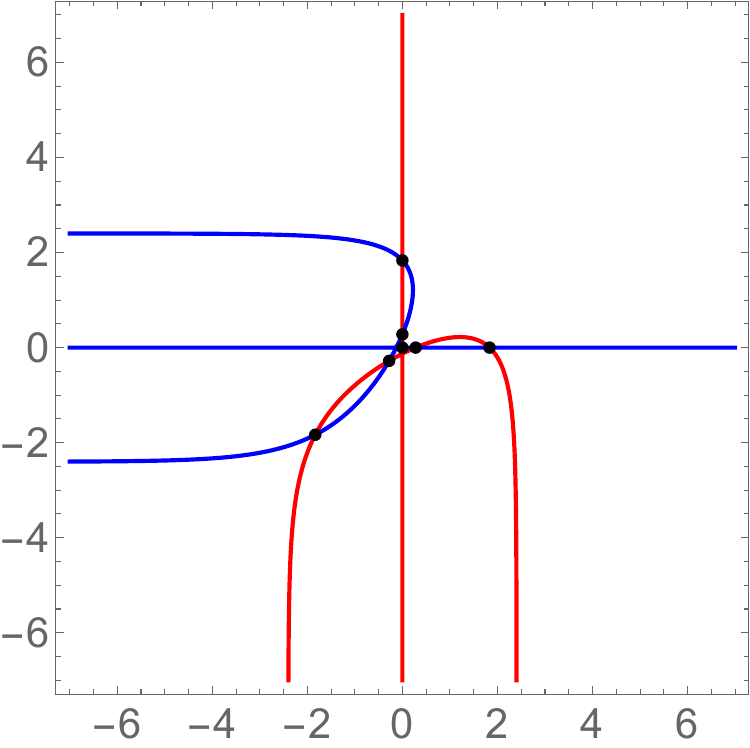}
\subcaption{$\gu=\gv=\psi=2.88$}
\label{fig:curves2}
\end{minipage}
\begin{minipage}{0.3\textwidth}
\centering 
\includegraphics[width=0.8\textwidth]{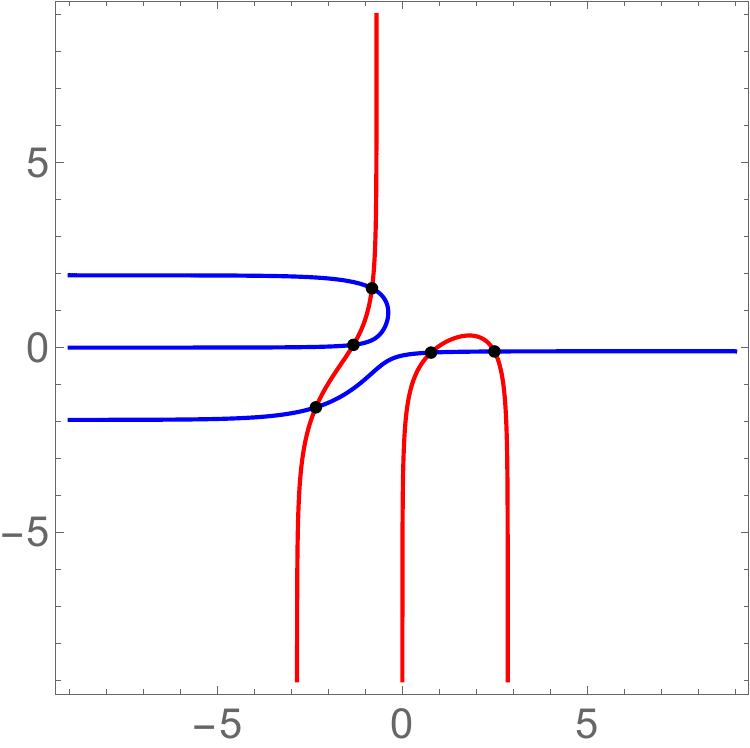}
\subcaption{$\gu=3.2, \gv=2.6, \psi=2.5$}
\label{fig:curves3}
\end{minipage}
\caption{
The curves $\curves_P$ (red) and $\curves_Q$ (blue) in the $(s,t)$-plane; their intersections $\critset$ correspond to the critical points of the density $\gm$. 
(\subref{fig:curves1}) Equilateral, equal-weight configuration with $\gu=\gv=\psi=3.92$: the central intersection at the origin is a critical point but not a local maximum, so no ghost mode is present. 
(\subref{fig:curves2}) Equilateral, equal-weight configuration with $\gu=\gv=\psi=2.88$ (Duistermaat configuration): the central intersection at the origin is now a local maximum, i.e. the ghost mode. 
(\subref{fig:curves3}) A configuration away from the equilateral case, $\gu=3.2, \gv=2.6, \psi=2.5$: in this particular asymmetric configuration, $(s,t)=(0,0)$ is no longer an intersection, and the central mode observed in panel (\subref{fig:curves2}) is not present. In all panels $w_1=w_2=w_3$, i.e., $\ku=\kv=0$. 
}
\label{fig:curves}
\end{figure}

\medskip

Although this direction is somewhat tangential to the main argument of the present paper, it is worth mentioning another possible representation of the equations. If $s$ is fixed, the equation $P(s,t)=0$ can be solved explicitly for $t$ as $t=\log \{(a_0(s)+a_1(s)e^s)/(-a_2(s))\}$, whenever the right-hand side is well defined. Conversely, if $t$ is fixed, solving $P(s,t)=0$ for $s$ leads to an expression in terms of a generalized $W$ function~\citep{mezo2017generalization}, which further generalizes the Lambert $W$ function~\citep{corless1996lambert,mezo2022lambert}. The Lambert $W$ function is closely related to the Pfaffian structure used in \citet{amendola2019maximum} and \citet{nguyen2026bounds}; see, for example, \citet{khovanskii1980class} for the concept of the Pfaffian structure. Such Pfaffian structures also allow one to bound the number of zeros of related functions, as studied by \citet{gabrielov2004Complexity}, building on Khovanskii's fewnomial theory~\citep{khovanskii1991fewnomials}. However, this route, employed in \citet{amendola2019maximum} and \citet{nguyen2026bounds}, is designed for substantially broader classes of mixtures; their direct application to the specific equations considered here may lead to rather loose estimates. For this reason, the present paper follows a different route, avoiding direct reliance on the general Pfaffian framework and instead exploiting the special exponential-polynomial structure of the reduced equations (Lemma~\ref{lem:exp_zeros}).

\subsection{Counting the number of intersections}
\label{subsec:counting_the_number_of_intersections}

We now bound the number of intersections of the curves $\curves_P$ and $\curves_Q$. This yields an upper bound on the number of critical points. 

Set $s_{\dagger}:=\psi-\gu-\ku$ and define
\[
N(s):=a_0(s)+a_1(s)e^s, \quad D(s):=-a_2(s)=s_{\dagger}-s,
\]
such that $P(s,t)=N(s)-D(s)e^t$. In what follows, we first distinguish the two cases $N(s_{\dagger})\ne0$ and $N(s_{\dagger})=0$. The distinction is needed because, in the case $N(s_{\dagger})=0$, the zero-counting problem develops a higher-order degeneracy at $s=s_{\dagger}$. 

In each case, we further divide the solutions according to whether $D(s)=0$. Since $D(s)=s_{\dagger}-s$, this is the same as separating the solutions off the line $s=s_{\dagger}$ from those lying on it.

\paragraph{Case A: $N(s_{\dagger}) \ne 0$}

\begin{enumerate}[{(1)}]
\item \textbf{Solutions with $D(s) \ne 0 \quad (\Leftrightarrow s \ne s_{\dagger})$} \\
In this case, $P(s,t)=0$ determines $t$ uniquely as
$t=g(s):=\log\{N(s)/D(s)\}$. This expression is defined precisely on
\[
U:=\left\{s \in \mathbb{R} \, \bigg| \, \frac{N(s)}{D(s)}>0, \, D(s) \ne 0\right\}.
\]

\paragraph{Strategy:} 
The idea is to reduce the intersection problem for the two curves $P(s,t)=0$ and $Q(s,t)=0$ to a univariate zero-counting problem. In Step~1, substituting the solution $t=g(s)$ of the equation $P(s,t)=0$ (defined over $U$) into $Q(s,t)=0$ gives a real-valued function $F:U\to\mathbb{R}$ whose zeros are in one-to-one correspondence with the critical points under consideration. In Step~2, we show that $U$ has at most two connected components. In Step~3, instead of counting the zeros of $F$ directly, we count the zeros of $F'$. For this purpose, we construct a function $G$ defined on all of $\mathbb{R}$ that agrees with $F'$ on $U$ up to multiplication by nonzero factors. The function $G$ is an exponential polynomial, so Lemma~\ref{lem:exp_zeros} gives an upper bound on the number of its zeros, and hence on the number of zeros of $F'$ in $U$. Finally, in Step~4, Rolle's theorem converts this bound on the zeros of $F'$ into a bound on the zeros of $F$, which gives the desired bound on the number of intersections.

In what follows, we describe the details of each step. 

\begin{itemize}
    \item \textbf{Step 1.} 
    In this step, we prove that the critical points in this case correspond bijectively to the zeros of some function $F(s)$ in $U$.

Since $P(s,t)=0 \Leftrightarrow a_2(s)e^t=-a_0(s)-a_1(s)e^s$, the remaining equation $Q(s,t)=0$ yields
\begin{align}
0
&=
a_2(s)Q(s,t) \nonumber \\
&=
a_2(s)b_0(t)+b_1(t)a_2(s)e^t+a_2(s)b_2(t)e^s \nonumber \\
&=
a_2(s)b_0(t)+b_1(t)\left\{-a_0(s)-a_1(s)e^s\right\}+a_2(s)b_2(t)e^s \nonumber \\
&\overset{(\star)}{=}
\bigl\{
\underbrace{a_2(s)b_0(t)-a_0(s)b_1(t)}_{=(\star 1)}\bigr\}
+
\bigl\{
\underbrace{a_2(s)b_2(t)-a_1(s)b_1(t)}_{=(\star 2)}\bigr\}e^s.
\label{eq:equation-qst}
\end{align}
Since each $b_j$ is affine with slope one, we have
$b_j(t)=t+b_j(0)$. Consequently,
\begin{align*}
(\star 1)
=
a_2(s)b_0(t)-a_0(s)b_1(t)
&=
(a_2(s)-a_0(s))t+a_2(s)b_0(0)-a_0(s)b_1(0),\\
(\star 2)
=
a_2(s)b_2(t)-a_1(s)b_1(t)
&=
(a_2(s)-a_1(s))t+a_2(s)b_2(0)-a_1(s)b_1(0).
\end{align*}
Considering $a_2(s)-a_0(s)=-\psi$ and $a_2(s)-a_1(s)=2\gu-\psi$, the
equation \eqref{eq:equation-qst} can be written as 
\[
L(s)t+M(s)=0.
\]
We next show that $L$ has no zero in $U$. Fix $s\in U$ and set
$t=g(s)$. Since 
\begin{align*}
    L(s)
    &=
    \{a_2(s)-a_0(s)\}+\{a_2(s)-a_1(s)\}e^s =
    a_2(s)(1+e^s)
    -a_0(s)-a_1(s)e^s
    =
    -D(s)(1+e^s)-N(s),
\end{align*}
we have
\begin{align*}
\frac{L(s)}{D(s)}
=
-\left( 1+e^s+\frac{N(s)}{D(s)} \right) <0,
\end{align*}
as $s \in U$ indicates $N(s)/D(s)>0$. Therefore, $L(s)\ne0$ throughout $U$.

Since multiplication by
$a_2(s)=-D(s)\ne0$ preserves the equation $Q(s,t)=0$, the above
calculation shows that
\[
Q(s,g(s))=0
\quad\Longleftrightarrow\quad
L(s)g(s)+M(s)=0
\quad\Longleftrightarrow\quad
\underbrace{g(s) + \frac{M(s)}{L(s)}}_{=: \, F(s)} = 0.
\]
Thus the critical points in Case A correspond bijectively to the
zeros of $F$ in $U$.

\item \textbf{Step 2.} 
In this step, we prove that $U$ has at most two connected components. 
Since $N(s)=a_0(s)+a_1(s)e^s$ and both polynomials $a_0,a_1$ are affine,
Lemma~\ref{lem:exp_zeros} shows that $N$ has at most three real zeros,
counted with multiplicity. The function $D(s)=s_{\dagger}-s$ has exactly
one simple zero. Hence $N(s)D(s)$ has at most four real zeros, counted
with multiplicity.

Moreover, $N(s)\to-\infty$ as $s\to-\infty$ and
$N(s)\to+\infty$ as $s\to+\infty$, whereas
$D(s)\to+\infty$ as $s\to-\infty$ and
$D(s)\to-\infty$ as $s\to+\infty$. Thus
$N(s)D(s)\to-\infty$ at both ends. Since
\[
U=\left\{s \in \mathbb{R} \, \mid \, N(s)D(s)>0\right\},
\]
it follows that $U$ has at most two connected
components. Indeed, each component has two boundary incidences at
zeros of $N(s)D(s)$, and a zero bordering two positive components has even
multiplicity and is counted at least twice.

\item \textbf{Step 3.} In this step, we count the zeros of the derivative $F'$. 
To simplify the calculation, define the real-analytic function on the whole real line:
\[
G(s)
:=
\{N'(s)D(s)-N(s)D'(s)\}L(s)^2
+
N(s)D(s)
\{L(s)M'(s)-L'(s)M(s)\},
\]
and it coincides with $F'(s)N(s)D(s)L(s)^2$ over $U$. 
Since $N$, $D$, and $L$ are nonzero on $U$, we have
$F'(s)=0$ if and only if $G(s)=0$ there. 
Collecting equal exponential terms in the displayed expression for
$G$ gives
\[
G(s)=\sum_{j=0}^3 g_j(s) \exp(j s),
\]
where the degrees of these polynomials $g_0,g_1,g_2,g_3$ satisfy
\[
\deg g_0 \le 2,
\quad
\deg g_1 \le 3,
\quad 
\deg g_2 \le 3, 
\quad 
\deg g_3 \le 2,
\]
where the coefficients of the leading terms (i.e., 2nd, 3rd, 3rd, and 2nd order terms) in $g_0,g_1,g_2,g_3$ are $2\gv \psi,4\gu \gv-\psi^2,4\gu\gv-\psi^2,-2(2\gu-\psi)(\gu+\gv-\psi)$, respectively. It shows that the coefficient of $s^2$ in $g_0$ is $2\gv\psi>0$, indicating that $g_0\not\equiv0$, and in particular $G\not\equiv0$. 
The direct algebraic expansion of each polynomial $g_j(s)$ is shown in Supplement S.1. 

Thus, Lemma~\ref{lem:exp_zeros} gives
\begin{align}
\#\{s \ne s_{\dagger} \, \mid \, F'(s)=0\}
&\le
\#\{s \in \mathbb{R} \, \mid \, G(s)=0\} \nonumber \\
&\le
\sum_{j=0}^{3} (\deg g_j+1) -1 
\le 
(3+4+4+3)-1
=
13,
\label{eq:zeros_of_G}
\end{align}
where the zeros are counted with multiplicity.

\item \textbf{Step 4.} In this step, we count the zeros of $F$ by applying Rolle's theorem. In any one connected component of $U$, any
$m$ distinct zeros of $F$ produce at least $m-1$ distinct zeros of
$F'$, and hence of $G$, between them. 
Since $U$ has at most two connected components as shown in Step 2, and $F'$ has at most $13$ zeros, Rolle's theorem indicates that any finite
collection of zeros of $F$ in $U$ has at most $13+2=15$ elements. 
\end{itemize}

Consequently, the entire zero set is finite and 
\[
\#\{(s,t)\in\critset \mid s\ne s_{\dagger}\}\le15.
\]

\item \textbf{Solutions with $D(s) = 0  \quad (\Leftrightarrow s = s_{\dagger})$}

Since $D(s_{\dagger})=s_{\dagger}-s_{\dagger}=0$, the equation $P(s_{\dagger},t)=0$ leads to $N(s_{\dagger})=0$ and is independent of $t$. If $N(s_{\dagger})\ne0$, there is no critical point on the line $s=s_{\dagger}$. Namely, 
\[
\#\{(s,t)\in\critset \mid s =  s_{\dagger}\} = 0.
\]

\end{enumerate}

\paragraph{Case B: $N(s_{\dagger}) = 0$}

\begin{enumerate}[{(1)}]
\item \textbf{Solutions with $D(s) \ne 0 \quad (\Leftrightarrow s \ne s_{\dagger})$}

Following Case A(1), we next consider the case $N(s_{\dagger})=0$. Since
$D(s_{\dagger})=s_{\dagger}-s_{\dagger}=0$, we can prove that both $L,M$ vanish at $s_{\dagger}$; 
in particular, $L(s_{\dagger})=-D(s_{\dagger})(1+e^{s_{\dagger}})-N(s_{\dagger})=0$ and 
$M(s_{\dagger})=-b_1(0)\{a_0(s_{\dagger})+a_1(s_{\dagger})e^{s_{\dagger}}\}=-b_1(0)N(s_{\dagger})=0$.

Put $\delta=s-s_{\dagger}$. Since $N,D,L$, and $M$ vanish at $s_{\dagger}$, there
exist real-analytic functions $\widetilde{N},\widetilde{D}\widetilde{L}$,
and $\widetilde{M}$ such that
$N=\delta \widetilde{N}$, $D=\delta \widetilde{D}$, $L=\delta \widetilde{L}$, and
$M=\delta \widetilde{M}$. Hence
$s_{\dagger}$ is a zero of 
\begin{align*}
G 
&=
(N'D-ND')L^2+ND(LM'-L'M)
=
\delta^4 \left\{ 
    (\widetilde{N}'\widetilde{D}-\widetilde{N}\widetilde{D}')\widetilde{L}^2
    +
    \widetilde{N}\widetilde{D}(\widetilde{L}\widetilde{M}'-\widetilde{L}'\widetilde{M})
\right\}
\end{align*}
of multiplicity at least four. 
Therefore, when
$N(s_{\dagger})=0$, the zeros of \(G\) in \(\mathbb{R}\setminus\{s_{\dagger}\}\) is at most \(13-4=9\).
In particular, \(G\) has at most \(9\) zeros in \(U\).

The same Rolle argument as in Step 4 of Case A1 then gives
\[
\#\{(s,t)\in\critset \mid s\ne s_{\dagger}\} \le
9+2=11.
\]

\item \textbf{Solutions with $D(s) = 0 \quad (\Leftrightarrow s = s_{\dagger})$}

Suppose that $N(s_{\dagger})=0$. The critical points on this line are then
the zeros in $t$ of $Q(s_{\dagger},t) = \{b_0(t)+b_2(t)e^{s_{\dagger}} \} + b_1(t)e^t$. 
This is an exponential polynomial with exponents $0$ and $1$, whose
corresponding polynomials have degree at most one. Moreover,
$b_1(t)=t-\gv+\kv$ is not identically zero. Hence
Lemma~\ref{lem:exp_zeros} shows that the line $s=s_{\dagger}$ contains at
most $(2+2)-1=3$ critical points. Namely, 
\[
\#\{(s,t)\in\critset \mid s =  s_{\dagger}\} \le 3.
\]

\end{enumerate}

It remains only to combine the estimates. Case A contributes at most $15$ critical points: subcase (1) gives at most $15$, and subcase (2) gives none. Case B contributes at most $14$ critical points: subcase (1) gives at most $11$, and subcase (2) gives at most $3$. Therefore,
\[
\#\critset
\le
\max\{\, \underbrace{15+0}_{\text{Case A}} \, , \, \underbrace{11+3}_{\text{Case B}} \, \}
=
15.
\]
This proves Theorem~\ref{theo:bound_critical_points}. \qed

\section{Conclusion}
\label{sec:conclusion}

This paper investigated the number of modes of three-component homoscedastic Gaussian mixtures. Existing general upper bounds are conditional on an a priori finiteness assumption and remain loose even in this simple setting, yielding bounds of $42592$ and $72$. We proved, without imposing any such additional assumption, that every three-component homoscedastic Gaussian mixture has at most $8$ modes.

\medskip \noindent
\textbf{Conjecture:} A natural next question is whether the bound obtained in this paper can be further sharpened. It is conjectured that the maximum number of modes of a general three-component Gaussian mixture is at most $6$ in ambient dimension $d=2$ (see \citet{amendola2019maximum} for historical context). In heteroscedastic settings, examples with $6$ modes can indeed be constructed rather easily~(see, e.g., \citet{amendola2019maximum} Figure~4). By contrast, our numerical experiments suggest that the homoscedastic case may be substantially more rigid; across a wide range of configurations, we did not find any example with more than $4$ modes. This observation leads us to conjecture that the sharp maximum in the three-component homoscedastic setting is $4$. Proving this sharper bound remains an important direction for future work.

\section*{Acknowledgements}

A. Okuno was supported by JSPS KAKENHI Grant Numbers 21K17718, 22H05106, and 25K03087. Y. Kabata was supported by  JSPS KAKENHI Grant Numbers 25H01485 and 25K00208. 
The authors thank Ryoya Fukasaku for helpful discussions. The authors also used Claude Opus 4.8, Claude Fable 5, ChatGPT-5.5 Pro, and ChatGPT-5.6 Sol Pro for mathematical discussion and language assistance.

\appendix

\section[Derivation of P(s,t) and Q(s,t)]{Derivation of \(P(s,t)\) and \(Q(s,t)\)}
\label{app:derivation}

Taking the logarithm of the first equation in \eqref{eq:simultaneous_simpler} gives
\[
\log \frac{\alpha}{1-\alpha-\beta}
=
\log \frac{w_2}{w_1}
+
L_u(\alpha,\beta)
=
-\ku + 2\alpha \gu + \beta \psi - \gu.
\]
Using the reparameterization \eqref{eq:ab_to_st} and its inverse \eqref{eq:st_to_ab}, this becomes
\[
s
+
\ku
-
2\frac{e^s}{1+e^s+e^t}\gu
-
\frac{e^t}{1+e^s+e^t}\psi
+
\gu
=
0.
\]
Multiplying both sides by $1+e^s+e^t$ yields $(s+\gu+\ku)(1+e^s+e^t) - 2e^s\gu - e^t\psi = 0$.
Rearranging the terms, we obtain
\[
(s+\gu+\ku)
+
(s-\gu+\ku)e^s
+
(s+\gu+\ku-\psi)e^t
=
0.
\]
Thus the left-hand side is precisely $P(s,t)=a_0(s)+a_1(s)e^s+a_2(s)e^t$. 
The equation $Q(s,t)=0$ is obtained in the same way from the second equation in \eqref{eq:simultaneous_simpler}.
\qed

\bibliographystyle{apalike}
\bibliography{GM}

@mastersthesis{wallace2013critical,
  author  = {Wallace, Benjamin},
  title   = {On the Critical Points of Gaussian Mixtures},
  school  = {Queen's University},
  address = {Kingston, Ontario, Canada},
  year    = {2013},
  type    = {M.Sc. thesis},
  note    = {Department of Mathematics and Statistics; listed by the department as an M.Sc. Project}
}

@phdthesis{amendola2017algebraic,
  author = {Am{\'e}ndola, Carlos},
  title  = {Algebraic Statistics of Gaussian Mixtures},
  school = {Technische Universit{\"a}t Berlin},
  year   = {2017},
  type   = {PhD thesis},
  doi    = {10.14279/depositonce-6557}
}

@book{coppel1971disconjugacy,
  author    = {William A. Coppel},
  title     = {Disconjugacy},
  series    = {Lecture Notes in Mathematics},
  volume    = {220},
  publisher = {Springer-Verlag},
  address   = {Berlin},
  year      = {1971}
}

@book{karlin1966tchebycheff,
  author    = {Samuel Karlin and William J. Studden},
  title     = {{Tchebycheff Systems: With Applications in Analysis and Statistics}},
  series    = {Pure and Applied Mathematics},
  volume    = {15},
  publisher = {Interscience Publishers},
  address   = {New York},
  year      = {1966}
}

@misc{nguyen2026bounds,
      title={Bounds on the Number of Modes of a Gaussian Mixture Density}, 
      author={Hien Duy Nguyen},
      year={2026},
      eprint={2605.15531v1},
      archivePrefix={arXiv},
      note={\url{https://arxiv.org/abs/2605.15531v1}}
}

@article{mezo2017generalization,
  title={On the generalization of the Lambert $W$ function},
  author={István Mező and Baricz, {\'A}rp{\'a}d},
  journal={Transactions of the American Mathematical Society},
  volume={369},
  number={11},
  pages={7917--7934},
  year={2017}
}

@article{ray2012upper,
title = {On the upper bound of the number of modes of a multivariate normal mixture},
journal = {Journal of Multivariate Analysis},
volume = {108},
pages = {41-52},
year = {2012},
author = {Surajit Ray and Dan Ren}
}

@article{kabata2025singularities,
  title={Singularities in bivariate normal mixtures},
  author={Kabata, Yutaro and Matsumoto, Hirotaka and Uchida, Seiichi and Ueki, Masao},
  journal={Information Geometry},
  volume={8},
  number={2},
  pages={343--357},
  year={2025},
  publisher={Springer}
}

@book{mclachlan2000finite,
  title={{Finite Mixture Models}},
  author={Geoffrey McLachlan and David Peel},
  series={Wiley Series in Probability and Statistics},
  year={2000},
  publisher={Wiley}
}

@book{mclachlan1988mixture,
  title={Mixture models. Inference and applications to clustering},
  author={McLachlan, Geoffrey J and Basford, Kaye E},
  publisher={Marcel Dekker},
  series={Statistics, textbooks and monographs},
  year={1988}
}

@book{mezo2022lambert,
  title={The {L}ambert $W$ function: {I}ts {G}eneralizations and {A}pplications},
  author={István Mező},
  year={2022},
  publisher={Chapman and Hall/CRC}
}

@inproceedings{gabrielov2004Complexity,
  title={Complexity of computations with Pfaffian and Noetherian functions},
booktitle={Normal forms, bifurcations and finiteness problems in differential equations},
  author={Andrei Gabrielov and Nicolai N. Vorobjov},
  year={2004},
volume={137},
pages={211-250},
series={NATO Science Series II: Mathematics, Physics and Chemistry}, 
publisher={Springer Dordrecht}
}

@book{khovanskii1991fewnomials,
  title={Fewnomials},
  author={Khovanskii, Askold Georgievich},
  series={Translations of Mathematical Monographs},
  volume={88},
  publisher={American Mathematical Society},
  year={1991},
  doi={doi.org/10.1090/mmono/088}
}

@article{khovanskii1980class,
  title={On a class of systems of transcendental equations},
  author={Khovanskii, Askold Georgievich},
  journal={Soviet Mathematics Doklady},
  volume={22},
  number={3},
  pages={762--765},
  year={1980},
  organization={American Mathematical Society}
}

@book{milnor1963morse,
  title={Morse Theory},
  author={John Willard Milnor},
  series={Annals of Mathematics Studies},
  year={1963},
  publisher={Princeton University Press}
}

@article{corless1996lambert,
  title={On the {L}ambert {W} function},
  author={Corless, Robert M and Gonnet, Gaston H and Hare, David EG and Jeffrey, David J and Knuth, Donald E},
  journal={Advances in Computational Mathematics},
  volume={5},
  pages={329--359},
  year={1996},
  publisher={Springer}
}

@article{damon1995local,
title = {Local Morse Theory for Solutions to the Heat Equation and Gaussian Blurring},
journal = {Journal of Differential Equations},
volume = {115},
number = {2},
pages = {368-401},
year = {1995},
author = {James Damon}
}

@article{Lifshitz1990multiresolution,
author = {Lifshitz, L. M. and Pizer, S. M.},
title = {A Multiresolution Hierarchical Approach to Image Segmentation Based on Intensity Extrema},
year = {1990},
publisher = {IEEE Computer Society},
address = {USA},
volume = {12},
number = {6},
issn = {0162-8828},
journal = {IEEE Transactions on Pattern Analysis and Machine Intelligence},
pages = {529–540},
numpages = {12}
}

@techreport{perpinan2003isotropic,
title = "An isotropic Gaussian mixture can have more modes than components",
author = "Carreira-Perpi{\~{n}}{\'a}n, Miguel {\'A}. and Williams, Christopher K. I.",
year = "2003",
month = dec,
day = "1",
language = "English",
type = "Working Paper",
institution = "The University of Edinburgh"
}

@InProceedings{perpinan2003number,
author="Carreira-Perpi{\~{n}}{\'a}n, Miguel {\'A}.
and Williams, Christopher K. I.",
editor="Griffin, Lewis D.
and Lillholm, Martin",
title="On the Number of Modes of a Gaussian Mixture",
booktitle="Scale Space Methods in Computer Vision",
year="2003",
publisher="Springer Berlin Heidelberg",
address="Berlin, Heidelberg",
pages="625--640"
}

@article{ray2005topography,
author = {Surajit Ray and Bruce G. Lindsay},
title = {{The topography of multivariate normal mixtures}},
volume = {33},
year={2005},
journal = {The Annals of Statistics},
number = {5},
publisher = {Institute of Mathematical Statistics},
pages = {2042 -- 2065}
}

@article{amendola2019maximum,
    author = {Am{\'e}ndola, Carlos and Engstr{\"o}m, Alexander and Haase, Christian},
    title = "{Maximum number of modes of Gaussian mixtures}",
    journal = {Information and Inference: A Journal of the IMA},
    volume = {9},
    number = {3},
    pages = {587-600},
    year = {2019}
}

@article{edelsbrunner2013isotropic, 
year = {2013}, 
title = {{Add Isotropic Gaussian Kernels at Own Risk: More and More Resilient Modes in Higher Dimensions}}, 
author = {Edelsbrunner, Herbert and Fasy, Brittany Terese and Rote, Gunter}, 
journal = {Discrete \& Computational Geometry}, 
pages = {797--822}, 
number = {4}, 
volume = {49}
}

\clearpage
\includepdf[pages=-]{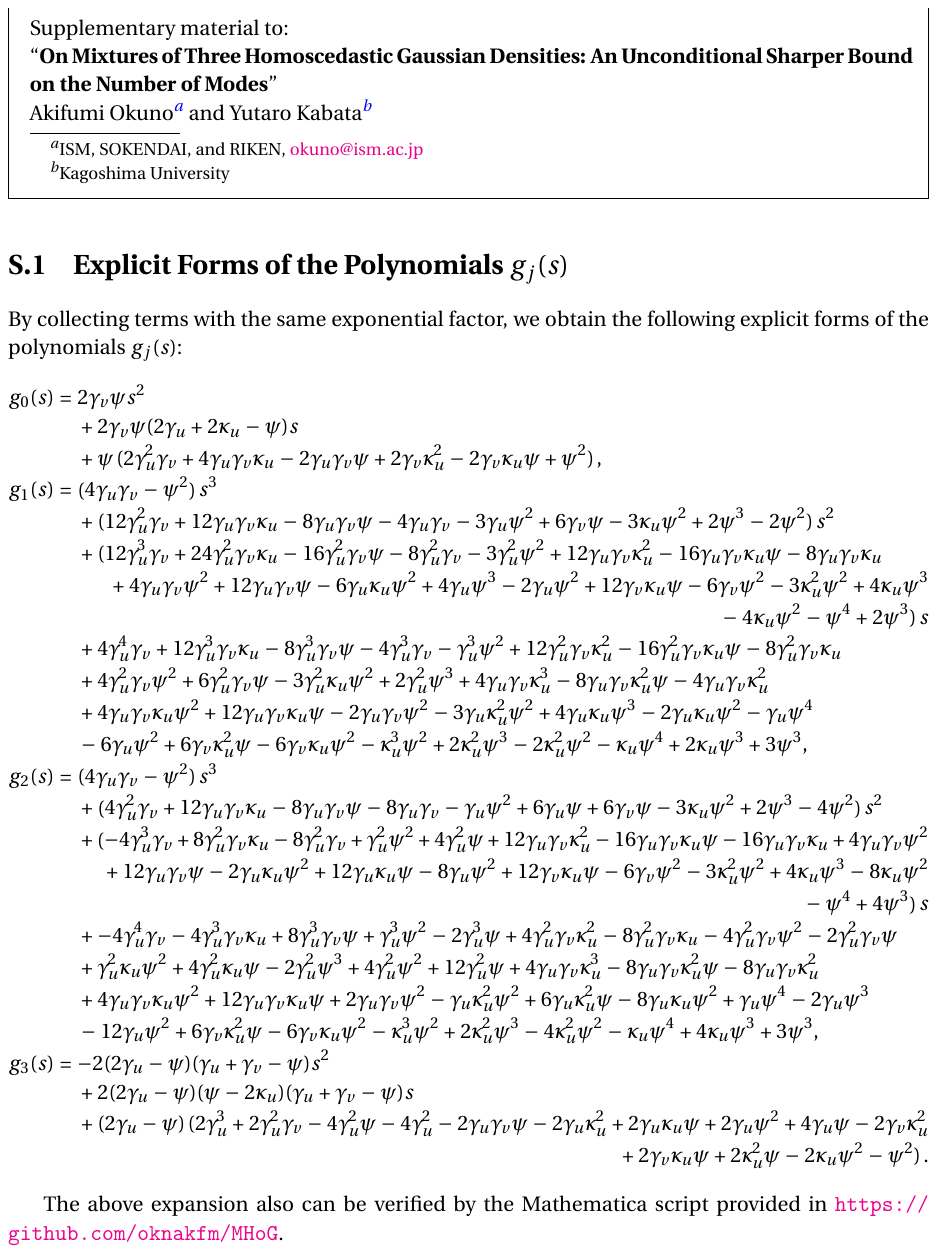}

\end{document}